\documentclass{amsart}
\usepackage{amssymb,amsrefs}
\usepackage[colorlinks]{hyperref}
\usepackage{booktabs}
\usepackage[all]{xy}
\usepackage{multirow}

\theoremstyle{plain}
\newtheorem{theorem}{Theorem}[section]

\newtheorem{lemma}[theorem]{Lemma}
\newtheorem{corollary}[theorem]{Corollary}

\theoremstyle{definition}
\newtheorem{remark}[theorem]{Remark}

\newtheorem{example}[theorem]{Example}

\DeclareMathOperator{\Aut}{Aut}
\DeclareMathOperator{\Aff}{Aff}
\DeclareMathOperator{\characteristic}{char}

\DeclareMathOperator{\Frob}{Frob}
\DeclareMathOperator{\id}{id}
\DeclareMathOperator{\GL}{GL}

\DeclareMathOperator{\Gal}{Gal}

\DeclareMathOperator{\Split}{Split}
\DeclareMathOperator{\rank}{rank}

\DeclareMathOperator{\ord}{ord}

\newcommand{\FF}{{\mathbf{F}}}
\newcommand{\OO}{\mathcal{O}}

\newcommand{\Que}{{\mathbf{Q}}}
\newcommand{\Zee}{{\mathbf{Z}}}
\newcommand{\Zhat}{{\widehat {\Zee}}}

\newcommand{\gothp}{{\mathfrak{p}}}
\newcommand{\gothq}{{\mathfrak{q}}}

\newcommand{\tto}{\longrightarrow}

\newcommand{\iso}{\cong}

\newcommand{\mapright}[1]{\mathop{\longrightarrow}\limits^{#1}}

\newcommand{\tor}{\textup{tor}}

\renewcommand{\hat}{\widehat}

\newcommand\lowtilde{\lower0.7ex\hbox{\textasciitilde}}

\newcommand{\mybar}[1]{
  \mathchoice
  {#1\llap{$\overline{\phantom{\displaystyle\rm#1}}$}}
  {#1\llap{$\overline{\phantom{\textstyle\rm#1}}$}}
  {#1\llap{$\overline{\phantom{\scriptstyle\rm#1}}$}}
  {#1\llap{$\overline{\phantom{\scriptscriptstyle\rm#1}}$}}
}  
\renewcommand{\bar}{\mybar}

\begin{document}

\title[Locally imprimitive points on elliptic curves]
{Locally imprimitive points on elliptic curves}

\author[Jones]{Nathan Jones} % 
\author[Pappalardi]{Francesco Pappalardi}
\author[Stevenhagen]{Peter Stevenhagen}
\address{Department of Mathematics, Statistics and Computer Science,
         University of Illinois at Chicago, 
         851 S Morgan St, 322 SEO,
         Chicago, IL 60607,
         USA}
        
\address{Dipartimento di Matematica,
         Universit\`a Roma Tre,
         Largo S. L. Murialdo 1,
         I--00146 Rome,
         Italy}
         
\address{Mathematisch Instituut,
         Universiteit Leiden, 
         Postbus 9512, 
         2300 RA Leiden, The Netherlands}
         
\email{ncjones@uic.edu, pappa@mat.uniroma3.it, psh@math.leidenuniv.nl}

\date{\today}


\begin{abstract}
Under GRH, any element in the multiplicative group of a number field $K$ 
that is globally primitive (i.e., not a perfect power in $K^*$) is 
a primitive root modulo a set of primes of $K$ of positive density.

For elliptic curves $E/K$ that are known to have
infinitely many primes~$\gothp$ of cyclic reduction, possibly under GRH,
a globally primitive point $P\in E(K)$ may fail to generate
any of the point groups $E(k_\gothp)$.
We describe this phenomenon in terms of an associated 
Galois representation $\rho_{E/K, P}:G_K\to\GL_3(\Zhat)$,
and use it to construct non-trivial examples
of global points on elliptic curves that are locally imprimitive.
\end{abstract}

\maketitle

\section{Introduction}
\label{S:intro}

\noindent
Under the Generalized Riemann Hypothesis (GRH), 
every non-zero rational number that is not $-1$ or a square is a primitive root modulo infinitely many primes~$p$.
This was proved in 1967 by Hooley \cite{Hooley},
forty years after Artin had stated it as a conjecture.
For general number fields $K$, there are counterexamples 
to the direct analogue of this statement, 
i.e., number fields $K$ with non-torsion elements $x\in K^*$ 
that are not $\ell$-th powers for any prime $\ell$ 
for which $K$ contains a primitive $\ell$-th root of unity, but that are nevertheless a primitive root in only 
finitely many residue class fields $k_\gothp$.
The direct analogue of Artin's conjecture does however hold
for $x\in K^*$ that are \emph{globally primitive}, i.e., 
not in ${K^*}^\ell$ for any prime $\ell$.
\begin{theorem}
\label{thm:mult-prim}
Let $K$ be a number field and $x\in K^*$ globally primitive, and assume GRH.
Then $x$ is a primitive root modulo $\gothp$ for a set primes $\gothp$ of $K$ of positive density.
\end{theorem}
\noindent
Informally stated: \emph{globally primitive} elements in $K^*$ are \emph{locally primitive} in $k_\gothp^*$ at infinitely many places $\gothp$,
with an element being `primitive' in $K^*$ or~$k_\gothp^*$ meaning that it generates a subgroup that 
is \emph{not} contained in a strictly larger cyclic subgroup of $K^*$ or~$k_\gothp^*$.
Section \ref{S:multiplicative primitivity} provides
a proof of Theorem \ref{thm:mult-prim}, and counterexamples to stronger statements.

Now replace the multiplicative group $K^*$ by the point group
$E(K)$ of an elliptic curve $E/K$, and the unit group $k_\gothp^*$ of the residue class field at $\gothp$ by the
point group $E(k_\gothp)$ at the primes of good reduction $\gothp$ of $E$.
Then we can add two natural questions to Artin's to obtain the following three problems:
for a number field $K$, determine the infinitude 
(or natural density) of the set of primes $\gothp$ in $K$ for which 
\begin{enumerate}
    \item[I.] (Artin) a given element $x\in K^*$ is a primitive root modulo $\gothp$, i.e., $k_\gothp^*=\langle\bar x\rangle$;
    \item [II.](Serre \cite{Serre3}) a given elliptic curve $E/K$ has cyclic reduction modulo $\gothp$;
    \item [III.](Lang-Trotter \cite{Lang-Trotter}) a given point $P\in E(K)$ generates the group $E(k_\gothp)$ at $\gothp$.
\end{enumerate}
We will refer to the cases I, II, and III as the 
\emph{multiplicative primitive root},
the \emph{cyclic reduction} and the 
\emph{elliptic primitive root} case, and denote the 
corresponding sets of primes by $S_{K, x}$, $S_{E/K}$, and 
$S_{E/K, P}$, respectively.
By definition, the finitely many primes of bad reduction for which we have 
$\ord_\gothp(x)\ne 0$ (in case I) or $\ord_\gothp(\Delta_E)\ne0$ (in case II and III, 
with $\Delta_E$ the discriminant of $E$) are excluded from these sets.
Note that we have an obvious inclusion 
$S_{E/K, P}\subset S_{E/K}$.

In each of the three cases, we have a group theoretical statement that
can be checked prime-wise at the primes $\ell$ dividing
the order of the groups $k_\gothp^*$ and $E(k_\gothp)$ involved. 
The statement `at $\ell$' has a translation in terms of the splitting behavior 
of $\gothp$ in a finite Galois extension $K\subset K_\ell$ 
that we describe in Section \ref{S:splittingconditions}.
Combining the requirements for all $\ell$ leads to (conjectural) density statements based on the
Chebotarev density theorem \cite{Stev-Lenstra}.

Imposing infinitely many splitting conditions, one for each prime $\ell$, 
leads to analytic problems with error terms that have been mastered 
under assumption of GRH in Cases I \cites{Hooley, Cooke-Weinberger, Lenstra} and II 
\cites{Serre3, Gupta-Murty-cycl, Campagna-Stevenhagen}, 
and that remain open in Case III.
In each case, there is a conjectural density 
$\delta_{K, x}$, $\delta_{E/K}$, or 
$\delta_{E/K, P}$ that is an upper bound for the upper density of the set of primes $\gothp$.
Proving unconditionally that the set is infinite in case 
the conjectural density is positive is an open problem.
If it is zero, we can however prove unconditionally that the corresponding set of primes is finite.

This paper focuses on the vanishing of $\delta_{E/K, P}$ 
in the elliptic primitive root case, which is much more subtle 
than the vanishing in the cases I and II. 
We call a global point $P\in E(K)$ \emph{locally imprimitive} 
if it is a generator of the local point group $E(k_\gothp)$ for 
only finitely many primes~$\gothp$ of $K$.
Our analysis will yield `elliptic counterexamples' $(E/K, P)$ to 
Theorem \ref{thm:mult-prim}, i.e., 
elliptic curves $E/K$ for which the cyclic reduction density 
$\delta_{E/K}$ is positive but for which
a globally primitive point $P\in E(K)$ is locally imprimitive.

Just as in the multiplicative primitive root and the cyclic reduction cases I and~II, obstructions to
local primitivity of a point $P\in E(K)$ become visible in an associated Galois representation.
In the elliptic primitive root case, the absolute Galois group  $G_K=\Gal(\overline K/K)$
of the number field $K$ acts on the subgroup of
$E(\overline K)$ consisting of the points $Q\in E(\overline K)$ satisfying  
$kQ\in\langle P\rangle\subset E(K)$ for some $k\in\Zee_{\ge1}$.
This yields a representation
$
\rho_{E/K, P}: G_K\to\GL_3(\Zhat).
$
Just as in the two other 
cases~\cite{Campagna-Stevenhagen},
it suffices to consider the residual representation
\begin{equation}
\label{residualrep}
\bar\rho_{E, P}: G_K\to\GL_3(\Zhat)\to\GL_3(\Zee/N\Zee)
\end{equation}
modulo a suitable squarefree integer $N$ that is divisible by all `critical primes'.

Unlike the cases I and II,
case III already allows non-trivial obstructions to local primitivity at prime level $N=\ell$.
In the multiplicative case I, the index $[k_\gothp^*:\langle\overline x\rangle]$ can only
be divisible by $\ell$ for almost all $\gothp$ for the `trivial reason' that
$K$ contains an $\ell$-th root of unity and $x$ is an $\ell$-th power in~$K^*$.
In the cyclic reduction case II, the group $E(k_\gothp)$ can only have a non-cyclic $\ell$-part 
for almost all~$\gothp$ for the `trivial reason' that the full $\ell$-torsion
of $E/K$ is $K$-rational.
In the elliptic primitive root case III however, there is a third reason why 
a point $P\in E(K)$ can be \emph{locally $\ell$-imprimitive}, 
meaning that $\ell$ divides the index
$[E(k_\gothp):\langle \overline P\rangle]$ for all but finitely many $\gothp$.
It is a less obvious one, and it was numerically discovered in 2015 in the case $\ell=2$ 
by Meleleo \cite{Meleleo}, who restricted himself to the basic case $K=\Que$.
\begin{theorem}
\label{thm:lneverprim}
Let $P\in E(K)$ be a non-torsion point of an elliptic curve $E$ defined over a number field $K$, 
and $\ell$ a prime number.
Then $P$ is locally $\ell$-imprimitive
if and only if at least at least one of the following conditions holds
\begin{enumerate}
\item[A.]
$E(K)$ contains a torsion point of order $\ell$ and  $P\in \ell\cdot E(K)$;
\item[B.]
$E$ has complete rational $\ell$-torsion over $K$;
\item[C.]
there exists an isogeny $\phi: E'\to E$ defined over $K$ with
kernel generated by a torsion point of $E'(K)$ of order $\ell$ and $P\in\phi[E'(K)]$.
\end{enumerate}
\end{theorem}

\noindent
Condition A in Theorem \ref{thm:lneverprim} 
is the analogue of the trivial condition from case~I:
if $E(K)$ has non-trivial $\ell$-torsion,
then almost all $E(k_\gothp)$ are groups of order divisible 
by~$\ell$, and for these $\gothp$
a point $P\in \ell\cdot E(K)$ will have its reduction in the subgroup 
$\ell E(k_\gothp) \subset E(k_\gothp)$ of index divisible by $\ell$.
Condition B bears no relation to $P$, and is well known
from case II:
non-cyclicity of the $\ell$-part of 
the global torsion subgroup $E(K)^\tor$
implies non-cyclicity of the $\ell$-part of $E(k_\gothp)$
at almost all $\gothp$.
At these $\gothp$, no single point $P$ can generate it.

Condition C has no analogue in the multiplicative primitive root case,
and it is a truly different condition 
as it includes cases in which we have both $P\notin \ell\cdot E(K)$ and $E[\ell](K)=0$.
If it holds, the dual isogeny $\hat\phi: E\to E'$
maps $P$ into $\ell E'(K)$, and the pair
$(E,P)$ is $\ell$-isogenous to the curve-point pair $(E',\hat\phi(P))$ satisfying Condition~A.
We call a locally 
$\ell$-imprimitive non-torsion point $P\in E(K)$ 
\emph{non-trivial} if Condition~C is satisfied, but \emph{not} Condition A or B.

\medskip\noindent 
By Theorem \ref{thm:lneverprim}, non-trivial locally 2-imprimitive points $P\in E(K)$
can only exist for $E/K$ having a single $K$-rational point of order 2, i.e., 
a 2-torsion subgroup of order $\#E(K)[2]\ne 1, 4$.
Examples of such points are actually surprisingly abundant.
\begin{theorem}
\label{thm:2neverprim}
Let $E/K$ be any elliptic curve with $\#E(K)[2]=2$. 
Then there are infinitely many quadratic twists of $E$ over $K$
that have a non-trivial locally $2$-imprimitive point.
\end{theorem}
\noindent
The proof of this Theorem, which we give in Section \ref{S:Locally 2-imprimitive points},
uses the fact that it is easy to create non-torsion points on twists of $E$, and exploits the particularly 
explicit description of $K$-rational 2-isogenies.

For primes $\ell>2$, it is harder to obtain families of elliptic curves 
with points of infinite order that are
locally $\ell$-imprimitive in non-trivial ways.
In Section \ref{S:Locally 3-imprimitive points} we provide an approach
in the case $\ell=3$.
It can be extended to higher values of $\ell$ 
(Section \ref{S:Further examples}), but the examples 
rapidly become unwieldy.

\medskip\noindent 
Non-torsion points that are locally imprimitive but not locally $\ell$-imprimitive for any single prime $\ell$ do exist, but they are not easily found.
They involve restrictions arising from reductions of $\rho_{E/K,P}$ of composite level caused by
non-trivial \emph{entanglement} between the fields~$K_\ell$.
In the context of the easier cyclic reduction case II,
this is discussed in \cite{Campagna-Stevenhagen}, and we present 
a first type of examples for our Lang-Trotter case III in our final Section \ref{S:A level 6 obstruction}.
Such higher level obstructions will be explored in more detail in a forthcoming paper.

\bigskip\noindent
{\bf Acknowledgements.}
\emph{
%We thank Hendrik Lenstra for his help with the general formulation of Lemma \ref{lemma:groupactions}.
All authors received support from the Max-Planck-Institut f\"ur Mathematik in Bonn while working on this paper.
They thank the institute for its financial support and for its very inspiring atmosphere.}

\vfil\eject
\section{Characterization by splitting conditions}
\label{S:splittingconditions}

\noindent
In each of the three cases discussed in the introduction, 
we can characterize the corresponding sets of primes
$S_{K, x}$, $S_{E/K}$, and $S_{E/K, P}$ of $K$
in terms of the splitting behaviour of their elements $\gothp$ 
in suitable extensions $K\subset K_\ell$, with $\ell$ ranging over all prime numbers.

\medskip\noindent
{\bf I. Multiplicative primitive root case.}
Let $K$ be a number field 
%$\mu_K\subset K^*$ the group of roots of unity in $K$ 
and $x\in K^*$ non-torsion.
Define $K_m=\Que(\zeta_m,\root m \of x)$ for $m\in\Zee_{\ge1}$
as the `$m$-division field of $x$', i.e., the splitting field over $K$ of the polynomial $X^m-x\in K[X]$.
If $\gothp$ is a prime of~$K$ of characteristic $p$ for which $x$ is 
a $\gothp$-adic unit, 
the index $[k_\gothp^*:\langle \overline x\rangle]$
is divisible by a prime $\ell\ne p$ if and only if $\gothp$ 
splits completely in $K\subset K_\ell$.

Note that $[k_\gothp^*:\langle \overline x\rangle]$ is 
never divisible by $p=\characteristic(\gothp)$, even though
$\gothp$ may split completely in $K_p$.
Example: $x=17$ is a primitive root modulo the prime
$\gothp_3$ of norm 3 in $K=\Que(\sqrt{-21})$, but $\gothp_3$
splits completely in the sextic extension
\[
K\subset K_3=
K(\zeta_3,\root 3\of {17})= K(\sqrt 7,\root 3\of {17}).
\]
This can however only happen for primes 
$\gothp|2\Delta_K$, with $\Delta_K$ the discriminant of $K$, since $K\subset K_p$ is ramified at all $\gothp|p$ 
for $p$ coprime to $2\Delta_K$.
In other words: for almost all $\gothp$,
the `condition at $\ell$' in the following Lemma is
automatically satisfied at $\ell=\characteristic{\gothp}$.
\begin{lemma}
\label{primrootatell}
For $\gothp$ a prime of $K$ outside the support of $x$, we have
$
k_\gothp^*=\langle \overline x\rangle $
%\ \  \Longleftrightarrow \ \ 
if and only if $\gothp$ does not split completely in
$K\subset K_\ell$ for any prime $\ell\ne\characteristic{\gothp}$.
\qed
\end{lemma}
\noindent
By Lemma \ref{primrootatell}, the set $S_{K, x}$ of primes in $K$ for which $x$ is a primitive root is up to finitely many primes equal to the set of primes that do not split completely in $K\subset K_\ell$
for any prime $\ell$.
For $m\in\Zee_{\ge1}$, the set of primes $\gothp$ of $K$ that split completely in $K\subset K_m$
has natural density $1/[K_m:K]$.
Under GRH, it follows from \cite{Lenstra} that the set $S_{K, x}$ has a natural density that is given by 
the inclusion-exclusion sum
\begin{equation}
\label{deltaKx}
\delta_{K,x}=
\sum_{m=1}^{\infty} \frac{\mu(m)}{[K_m:K]}
\end{equation}
that converges slowly, but that can be rewritten in `factored form' as
\begin{equation}
\label{deltaKx-factored}
\delta_{K,x}=
\sum_{m|N} \frac{\mu(m)}{[K_m:K]}\cdot
   \prod_{\ell\nmid N \text{ prime}} (1-\frac{1}{\ell(\ell-1)}).
\end{equation}
Here we can take for $N$ any integer 
divisible by the primes in some finite set of \emph{critical primes}.
One may take for this set the set of primes that are either in the support of $x$ or 
divide $2\Delta_K$,  
together with those primes $\ell$ for which $x$ is in~${K^*}^\ell$.
The essential feature of $N$ is that the family 
$\{K_\ell\}_{\ell\nmid N}$ of `$\ell$-division fields of $x$
outside~$N$' is a linearly disjoint family over $K$ 
with each $K_\ell$ having the full Galois group 
$\Gal(K_\ell/K)\cong \Aff_1(\FF_\ell)=\FF_\ell\rtimes \FF_\ell^*$
of order $\ell(\ell-1)$, and that the compositum $L$ 
of the fields in this family satisfies $L\cap K_N=K$.

\medskip\noindent
{\bf II. Cyclic reduction case.}
For an elliptic curve $E/K$, we consider the set $S_{E/K}$ of 
primes of cyclic reduction of $E$, i.e.,
the primes $\gothp$ of $K$
for which $E$ has good reduction and the reduced elliptic
curve point group $E(k_\gothp)$ is cyclic.
The condition that $E$ have good reduction modulo $\gothp$ 
only excludes the finitely many primes dividing the
discriminant $\Delta_E$ of $E$.
 
For $m\in\Zee_{\ge1}$, we define $K_m=K(E[m](\overline \Que))$ 
in this case to be the $m$-division field of $E$ over~$K$.
The following elementary lemma \cite{Campagna-Stevenhagen}*{Corollary 2.2} 
is the analogue of Lemma \ref{primrootatell}. 
It expresses the fact that a finite abelian group is cyclic if and only if 
its $\ell$-primary part is cyclic for all primes $\ell$.

\begin{lemma}
\label{cyclicreductionatell}
A prime $\gothp$ of good reduction of $E/K$
is a prime of cyclic reduction
if and only if $\gothp$ does not split completely in 
$K\subset K_\ell$ for any prime 
$\ell\ne\characteristic{\gothp}$.
\qed
\end{lemma}
\noindent
As in the multiplicative Case I, cyclicity of the $p$-primary part of the groups $E(k_\gothp)$
is automatic for $p=\characteristic (\gothp)$.
Also here, total splitting of $\gothp$ in non-trivial extensions $K\subset K_p$ for 
$p=\characteristic{\gothp}$ does occur: it suffices to base change any elliptic curve $E/K$ with 
$K\subset K_p=K[X]/(f)$ non-trivial by an extension $K\subset L= K[X]/(g)$ with 
$f$ and $g$ polynomials of the same degree that are $\gothp$-adically close, 
but with $g$ Eisenstein at a prime $\gothq$ that is unramified in $K\subset K_p$.
For $E/L$, the non-trivial extension $L\subset L_p$ will be totally split at all primes dividing $\gothp$.
Again, this can happen only at primes $\gothp|2\Delta_K$, as otherwise $K\subset K_p$ will be
ramified at all primes $\gothp$ of characteristic $p$ by the fact that $K_p$ contains $\zeta_p$.
Thus, for almost all~$\gothp$,
the `condition at $\ell$' in Lemma \ref{cyclicreductionatell} is again 
automatically satisfied at $\ell=\characteristic{\gothp}$.
The finitely many primes dividing $2\Delta_K$ are clearly irrelevant when dealing with the density of the set $S_{E/K}$, which, just like in the previous case, coincides up to 
finitely many primes with the set of primes $\gothp$ of $K$ that do not 
split completely in $K\subset K_\ell$ for any prime $\ell$.

Under GRH, the density of $S_{E/K}$ is again given 
\cite{Campagna-Stevenhagen}*{Section 2} by an
inclusion-exclusion sum that we already know from 
\eqref{deltaKx}:
\begin{equation}
\label{deltaEK}
\delta_{E/K}=\sum_{m=1}^{\infty} \frac{\mu(m)}{[K_m:K]}.
\end{equation}
If $E$ is without CM over $\overline\Que$,
or has CM by an order $\OO\subset K$, there is in each case a factorization of $\delta_{E/K}$ that is typographically identical to
\eqref{deltaKx}, provided that $N$ is divisible 
by all  primes from an appropriately defined finite set of critical primes
\cite{Campagna-Stevenhagen}*{Theorems 1.1 and 1.2}.
If $E$ has CM by an order $\OO\not\subset K$, there is a
hybrid formula 
\cite{Campagna-Stevenhagen}*{Theorem 1.4}
with different contributions from ordinary and supersingular primes.

A `factorization formula' for $\delta_{K, x}$ and 
$\delta_{E/K}$ as in \eqref{deltaKx-factored} shows that
the vanishing of these densities is always caused
by an obstruction at some finite level $N$. 
For such $N$, no element in $\Gal(K_N/K)$ restricts for all $\ell|N$ to a non-trivial 
element of $\Gal(K_\ell/K)$. 
As a consequence, there are no non-critical primes in $S_{K,x}$ or~$S_{E/K}$: 
the Frobenius elements of such primes in $\Gal(K_N/K)$ 
cannot exist for group theoretical reasons.

An obstruction at prime level $N=\ell$, which means 
an equality $K=K_\ell$, amounts in the cases I and II to a `trivial' reason
that we already mentioned in the context of Conditions A and~B in 
Theorem \ref{thm:lneverprim}.
For $K=\Que$, vanishing of $\delta_{\Que, x}$ and $\delta_{E/\Que}$ only occurs if we have $K=K_\ell$ for a prime $\ell$, which in this case has to be $\ell=2$.

Over general number fields, vanishing may be caused
by obstructions that occur \emph{only} at composite levels.  
Typical examples can be constructed by base changing a
non-vanishing example $(K, x)$ or $E/K$ to a suitable extension field $K\subset L$.
Example \ref{ex:counterexample} accomplishes vanishing of $\delta_{K,x}$ 
by an obstruction at level $30=2\cdot 3\cdot 5$ for the field $K=\Que(\sqrt 5)$ 
that does not arise at lower level by cleverly choosing $x$.
In \cite{Campagna-Stevenhagen}*{Example 5.4} we find 
a base change of an elliptic curve $E/\Que$ to a field $K$ of degree 48
with a similar level 30 obstruction to cyclic reduction.

\medskip\noindent
{\bf III. Elliptic primitive root case.}
In addition to the elliptic curve $E/K$, we are now given a
point $P\in E(K)$ of infinite order.
We consider the set $S_{E/K, P}$ of primes $\gothp$ of $K$ for which $E$ has cyclic reduction and 
the reduction of the point $P$ modulo~$\gothp$ generates the group $E(k_\gothp)$. 
Note the obvious inclusion $S_{E/K, P}\subset S_{E/K}$.

For $m\in\Zee_{\ge1}$, we let $K_m=K(m^{-1}P)$ be the 
\emph{$m$-division field  of $P$}, i.e., the extension of $K$ 
generated by the points of the subgroup
of $E(\overline K)$ defined as
\[
\langle m^{-1} P\rangle = 
\{Q\in E(\overline K): mQ \in \langle P\rangle\}.
\]
Note that this extension $K_m$ contains the $m$-division field of the elliptic curve $E$
that we encountered in the cyclic reduction case II.
The $m$-division field $K_m=K(m^{-1}P)$ of $P$ is again unramified over $K$ 
at primes $\gothp$ of good reduction coprime to $m$. 
The proof of this fact is as for the $m$-division field of $E$: 
as the $m$-th roots $Q\in E(\overline K)$ of $P$ that generate $K_m$ over $K$ 
differ by $m$-torsion points,
their reductions modulo a prime over $\gothp$ remain different, 
so inertia acts trivially on the set of such~$Q$.

The quotient group 
$V_m=\langle m^{-1} P\rangle / \langle P\rangle$
is a free module of rank 3 over $\Zee/m\Zee$.
It comes with a natural linear action of the absolute Galois group $G_K$ of $K$, and this 
mod-$m$ Galois representation induces an embedding
\begin{equation}
\label{mod-ell-rep}
G_m=\Gal(K_m/K)\hookrightarrow \GL(V_m)\iso \GL_3(\Zee/m\Zee).
\end{equation}
As $G_K$ stabilizes the rank 2 subspace 
$U_m=E[m](\overline K)\subset V_m$, this is a `reducible' representation.
Write $V_m=U_m\oplus (\Zee/m\Zee)\cdot \bar Q$ with $\bar Q = (Q\bmod \langle P\rangle) \in V_m$
the image of a point $Q\in E(\overline K)$ satisfying $m Q=P$.
We then have a split exact sequence 
\[
0\to U_m \tto  V_m \tto \Zee/m\Zee\cdot \bar Q \to 0
\]
of free $\Zee/m\Zee$-modules that is split as a sequence of 
$(\Zee/m\Zee)[G_m]$-modules if and only if we have
$P\in m \cdot E(K)$.
After choosing an $\Zee/m\Zee$-basis $\{T_1, T_2\}$ of $U_m$ 
and extending it by some $\bar Q$ as above to a 
$\Zee/m\Zee$-basis for $V_m$,
the matrix representation of $\sigma\in G_m$ becomes
\begin{equation}
\label{matrixrep}
\sigma=\sigma_{A,b}=
\begin{pmatrix}
a_{11} & a_{12} & b_1 \\
a_{21} & a_{22} & b_2 \\
0 & 0 & 1 
\end{pmatrix},
\end{equation}
in which 
the linear action of $\sigma$ on $U_m$ with respect to 
some $\Zee/m\Zee$-basis $\{T_1, T_2\}$ of~$U_m$
is described by
$A=\begin{pmatrix}
a_{11} & a_{12}  \\
a_{21} & a_{22}  
\end{pmatrix}
\in \GL(U_m)$,
and 
\[
b=b_1T_1+b_2T_2=\sigma Q - Q
\]
is the translation action of $\sigma$ on some chosen `$m$-th root' 
$Q$ of $P$ with respect to that same basis.
In other words, $V_m$ gives a Galois representation of $G_K$ with image $G_m\subset \GL(V_m)$
that is contained in the 2-dimensional
affine group
\[
\Aff_2(\Zee/m\Zee)=(\Zee/m\Zee)^2 \rtimes \GL_2(\Zee/m\Zee).
\]
In the important case where $m=\ell$ is prime, we are in the classical situation
of a 3-dimensional Galois representation over the finite field $\FF_\ell$.

The analogue in the elliptic primitive root case of the 
Lemmas \ref{primrootatell} and \ref{cyclicreductionatell}
is a little bit more involved.
%rather than requiring that $\gothp$ do not split completely in extensions $K\subset K_\ell$
%(which corresponds to a different question that has also been studied \cite{Virdol}), 
We have to impose a condition on the Frobenius elements 
$\Frob_{\gothp, \ell}\in G_\ell$ at all primes $\ell\ne\characteristic{\gothp}$
different from being \emph{equal} to the identity element $\id_\ell\in G_\ell$: 
in this case it only needs to be `sufficiently close' to it.
\begin{lemma}
\label{ellprimrootatatell}
For $P\in E(K)$ of infinite order and $\gothp$ a prime of good reduction of~$E$ of characteristic different from $\ell$ we have 
\[
\ell | [E(k_\gothp):\langle {\bar P}\rangle ]
\quad \Longleftrightarrow \quad 
\rank (\Frob_{\gothp,\ell} - \id_\ell)\le 1.
%\text{ for all primes $\ell\ne\characteristic{\gothp}$}.
\]
\end{lemma}
%\noindent
\begin{proof}
As all $V_\ell$ are 3-dimensional over $\FF_\ell$, the condition $\rank (\Frob_{\gothp,\ell} - \id_\ell)\le 1$
means that $\Frob_{\gothp,\ell}$ is the identity on an $\FF_\ell$-subspace of $V_\ell$ of dimension at least~2.
If it equals~$U_\ell$, then $E(k_\gothp)$ has complete $\ell$-torsion 
of order $\ell^2$ and every cyclic subgroup $\langle {\bar P}\rangle$ has index divisible by $\ell$.
If not, it intersects $U_\ell$ in a 1-dimensional subspace, so we have a point of order $\ell$ 
in $E(k_\gothp)$ and a point $\bar Q\in E(k_\gothp)$ satisfying $\ell\bar Q=\bar P$.
This also implies that $[E(k_\gothp):\langle {\bar P}\rangle]$ is divisible by $\ell$.

Conversely, if $\ell$ divides $[E(k_\gothp):\langle {\bar P}\rangle]$ then either 
$E(k_\gothp)$ has complete $\ell$-torsion or $E(k_\gothp)$ has a cyclic non-trivial $\ell$-part and
$\bar P$ is contained in $\ell\cdot E(k_\gothp)$. 
In both cases $\Frob_{\gothp,\ell}$ is the identity on a subspace of $V_\ell$ of dimension at least 2.
\end{proof}

\begin{corollary}
\label{LTcondition}
Let $P\in E(K)$ be of infinite order and $\gothp$ a prime of 
good reduction of $E$ of prime norm $\characteristic{\gothp}>5$ 
for which $\bar P\ne \bar O\in E(k_\gothp)$.
Then we have
\[
E(k_\gothp)=\langle {\bar P}\rangle
\quad \Longleftrightarrow \quad 
\rank (\Frob_{\gothp,\ell} - \id_\ell)\ge 2
\text{ for all primes $\ell$}.
\]
\end{corollary}
\begin{proof}
By Lemma \ref{ellprimrootatatell}, the condition on the right 
side says that $p=\characteristic{\gothp}$ is the only possible
prime divisor of the index 
$[E(k_\gothp):\langle {\bar P}\rangle ]$.
For a prime $\gothp$ of degree one, i.e., of prime norm $p$,
the index of a subgroup
of $E(k_\gothp)=E(\FF_p)$ can only be divisible by $p$ 
if it is the trivial subgroup, as we have 
$\#E(\FF_p)< p+1+2\sqrt p<2p$ for $p>5$.
So we have $E(k_\gothp)=\langle {\bar P}\rangle$ unless 
$\gothp$ is a prime for which we have 
$\bar P = \bar O\in E(k_\gothp)$.
As we have $P\ne O\in E(K)$, this happens only for finitely 
many $\gothp$.
\end{proof}
\noindent
In density questions, we can disregard any finite set of primes, and more generally a set of primes of density zero.
The set of primes in a number field of degree bigger than one is such a zero density set.
For this reason, the density of the set
$S_{E/K,P}$ only depends on the primes of degree one outside any finite set of `critical primes' that it contains.
Thus, Corollary \ref{LTcondition} can play the same role as the 
Lemmas \ref{primrootatell} and~\ref{cyclicreductionatell}.

In order to express the `heuristical density' 
$\delta_{E/K, P}$ of $S_{E/K,P}$,
we define the subset 
$S_\ell\subset G_\ell=\Gal(K_\ell/K)$ 
of `bad' elements at the prime $\ell$ as
\[
S_\ell=\{\sigma\in G_\ell: 
\rank_{\FF_\ell} (\sigma - \id_\ell)\le 1\}.
\]
For arbitrary $m\in\Zee_{\ge1}$ and $\ell|m$ prime we let $\pi_{m,\ell}:G_m\to G_\ell$ be the natural
restriction map, and define $S_m\subset G_m$ as 
\[
S_m=\bigcup_{\ell|m}  \pi_{m,\ell}^{-1}[S_\ell].
\]
With $s_m=\#S_m$ denoting the cardinality of $S_m$, the \emph{elliptic primitive root density}
is now given by the inclusion-exclusion sum  
\begin{equation}
\label{deltaEP}
\delta_{E,P}=
\sum_{m=1}^{\infty} \frac{\mu(m)s_m}{[K_m:K]}.
\end{equation}
It is the elliptic analogue of the multiplicative primitive root density \eqref{deltaKx}.
It is an upper density for $S_{E/K,P}$ that has not been proven to be its true density 
in cases with $\delta_{E/K, P}>0$, not even under GRH.

We can compute $\delta_{E/K, P}$ using the methods of \cite{Campagna-Stevenhagen}.
This is not directly relevant for us, as our focus in this paper is on cases
where $\delta_{E/K, P}$ vanishes in `non-trivial' ways, so we will merely sketch this here.
In order to obtain a factorization 
\begin{equation}
\label{deltaEP-factored}
\delta_{E/K,P}=
\sum_{m|N} \frac{\mu(m)s_m}{[K_m:K]}\cdot
   \prod_{\ell\nmid N \text{ prime}} (1-\frac{s_\ell}{[K_\ell:K]}).
\end{equation}
as in \eqref{deltaKx-factored}, it suffices to have 
an `open-image theorem' for the Galois representation $\rho_{E,P}$ 
arising from the action of $G_K$ on the subgroup
\[
R_P=\{Q\in E(\bar K): m Q\in \langle P\rangle
\text{ for some } m\in \Zee_{\ge1}\} \cong (\Que/\Zee)^2\times \Que
\]
of $E(\bar K)$ generated by all the roots of $P$ in $E(\bar K)$.
The Galois action of $G_K$ on the quotient group $V=R_P/\langle P\rangle$, 
which is free of rank 3 over $\Que/\Zee$, gives rise to a Galois representation 
\[
\rho_{E,P}: G_K \tto \Aut(V)\iso \GL_3(\Zhat),
\]
which has \eqref{mod-ell-rep} as its mod-$m$ representation.
It factors via $\Gal(K(R_P)/K)$, with $K(R_P)=\bigcup_m K_m\subset \bar K$ 
the compositum of all `$m$-division fields' of $P$ inside $\bar K$.
The group $U=E(\overline K)^\tor\iso (\Que/\Zee)^2$ is a direct summand of $V$,
and if we choose a $\Que/\Zee$-basis for $V=U\oplus \Que/\Zee$ 
as we did for $V_m=U_m\oplus \Zee/m\Zee\cdot\bar Q$,
the image of $\rho_{E,P}$ is in $\Aff_2(\Zhat)=\Zhat^2\rtimes \GL_2(\Zhat)$.
For $E$ without CM over $\bar K$, one deduces from Serre's open image theorem
that this image is of finite index in $\Aff_2(\Zhat)$, 
which yields \eqref{deltaEP-factored} for any $N$ divisible by some
finite product $N_{E/K,P}\in\Zee_{>0}$ of critical primes.
As in~\cite{Campagna-Stevenhagen}, one deduces that all non-CM-densities $\delta_{E,P}$
are rational multiples of a universal constant.
If $E$ has CM over $\bar K$ by an order $\OO\subset K$, 
one replaces $\Aff_2(\Zhat)$ by $\Aff_1(\OO)$,
and in the case of CM by an order $\OO\not\subset K$, one separates 
the contribution of ordinary and supersingular primes of $E/K$ 
as in \cite{Campagna-Stevenhagen}.

\section{Multiplicative primitivity}
\label{S:multiplicative primitivity}

\noindent
Before focusing on the Lang-Trotter case III, 
we first settle the multiplicative primitive root case: under GRH,
globally primitive elements $x\in K^*$ are locally primitive
for a set of primes of positive density $\delta_{K,x}$.

\begin{proof}[Proof of Theorem \ref{thm:mult-prim}]
Let $x\in K^*$ be globally primitive.
As we assume GRH, the primitive root density for $x\in K^*$ exists 
and is equal to $\delta_{K,x}$ defined in~\eqref{deltaKx}, by the results of \cite{Lenstra}.
We need to show that $\delta_{K,x}$ does not vanish.
In view of the factorization 
formula~\eqref{deltaKx-factored}, it suffices to show that 
for any squarefree integer $N>1$, the fraction
$\sum_{m|N} \mu(m)[K_m:K]^{-1}$ of elements in 
$\Gal(K_N/K)$ that have non-trivial restriction to $K_\ell$ for all primes $\ell|N$ does not vanish.

As $x$ is not an $\ell$-th power in $K^*$, the polynomial $X^\ell-x$ is irreducible in $K[X]$.
It therefore gives rise to an extension 
$K\subset K_\ell=\Split_K(X^\ell-x)=K(\zeta_\ell,\root \ell \of x)$ 
of degree $\ell\cdot c_\ell$, with $c_\ell$ a divisor of $\ell-1$. 
If $\ell$ is the largest prime dividing the squarefree number $N$, we conclude that 
$K_{N/\ell}\subset K_N$ is Galois of degree divisible by  $\ell$.

Showing that $\Gal(K_N/K)$ contains an element of the required type is now easily done by induction on the number of of primes dividing the squarefree integer $N>1$. 
If $N$ is prime, then $\Gal(K_N/K)$ contains a non-trivial element of order $N$.
If not, we let $\ell$ be the largest prime dividing $N$ and observe that
an automorphism of the required type in $\Gal(K_{N/\ell}/K)$, 
which exists by the induction hypothesis,
always possesses an extension to the compositum $K_N$ of $K_{N/\ell}$ and $K_\ell$ that is non-trivial on $K_\ell$.
\end{proof}

\noindent
The assumption of global primitivity in Theorem \ref{thm:mult-prim}
cannot be weakened to the assumption $K\ne K_\ell$ for all prime number $\ell$.
The resulting stronger statement is correct for $K=\Que$, 
but counterexamples to it exist for general number fields $K$, 
as the cyclotomic extensions $K\subset K(\zeta_\ell)$ for different $\ell$ may all be non-trivial, but `entangled' over $K$.
The following counterexample takes $K$ to be quadratic.
\begin{example}
\label{ex:counterexample}
The quadratic field $K=\Que(\sqrt 5)$ has 
fundamental unit  $\varepsilon=\frac{1+\sqrt 5}{2}$.
The element $\pi=\varepsilon^2-4=\frac{-5+\sqrt 5}{2}\in K$
has norm 5 and is a square modulo~4.
The field $K(\sqrt\pi)$, which is cyclic of degree 4 
over $\Que$ and unramified outside 5, is therefore equal to $K(\zeta_5)$.
Take $y=-3\pi\in K$ and choose $x=y^{15}$.
We then have 
\[
K_3=K(\zeta_3)=K(\sqrt{-3})\qquad\text{and}\qquad
K_5=K(\zeta_5)=K(\sqrt\pi),
\]
so $K_2=K(\sqrt x)=K(\sqrt y)=K(\sqrt{-3\pi})$ and $K_3$ 
and $K_5$ are three different quadratic extensions of $K$ contained in 
the biquadratic extension $K\subset K_6=K_{10}=K_{15}=K_{30}$.
We have $\mu_K=\{\pm1\}$ and, even though $x$ is not a 
square in $K^*$, there is exactly one prime of $K$ 
modulo which $x$ is a primitive root: $(2)$.
For the primes $\gothp=(3)$ and $\gothp=(\sqrt 5)$ 
the element $x$ is not in $k_\gothp^*$, and 
for all primes of characteristic $p>5$
the index $[\FF_p:\langle\bar x\rangle]$ is divisible by at least one of 2, 3 or 5.
Indeed, no prime can be inert in all three quadratic subfields of 
an extension with group $\Zee/2\Zee\times\Zee/2\Zee$.
\end{example}

\medskip\noindent
The simple observation that no prime $\gothp$ of a number field $K$ can be inert in all three quadratic subextensions of an extension $K\subset L$ with group $\Zee/2\Zee\times\Zee/2\Zee$ underlies many `entanglement obstructions', including the one in our 
final Section \ref{S:A level 6 obstruction}.
\section{Proof of Theorem \ref{thm:lneverprim}}
\label{S:Proof of Theorem 1.2}

\noindent
In the Lang-Trotter situation, 
Lemma \ref{ellprimrootatatell} shows that
a point $P\in E(K)$ will generate a subgroup 
of the local point group $E(k_\gothp)$ of index 
divisible by $\ell$ when $\Frob_{\gothp, \ell}\in G_\ell$
pointwise fixes a 2-dimensional subspace of 
the 3-dimensional $\FF_\ell$-vector space $V_\ell$.
Vanishing of the density $\delta_{E/K,P}$ can therefore occur 
`because of $K_\ell$' in cases where 
$G_\ell=\Gal(K_\ell/K)$ is non-trivial, 
but only contains elements that pointwise fix
a 2-dimensional subspace of $V_\ell$.

Our proof of Theorem \ref{thm:lneverprim}
is based on a general lemma that we phrase and prove in the generality
that was suggested to us by Hendrik Lenstra.
It describes the linear group actions on vector spaces of finite 
dimension over arbitrary fields that have `many fixpoints' 
in the sense of the 3-dimensional example $V_\ell$ that we have at hand.

\medskip\noindent
Let $V$ be any vector space on which a group $G$ acts linearly, and denote by
$$
V^G=\{v\in V: \sigma v=v \text{ for all }\sigma \in G\}
\qquad\text{and}\qquad
V_G= V/(\textstyle\sum_{\sigma\in G} (\sigma-1)V)
$$
the maximal subspace and quotient space of $V$
on which $G$ acts trivially.
For every $\sigma \in G$, we have an exact sequence of vector spaces
$$
0\tto V^{\langle\sigma\rangle} \tto V \mapright {\sigma-1} V
\tto V/(\sigma-1)V\to 0
$$
showing that for $V$ of finite dimension $n$,
we have 
\begin{equation}
\label{invariants-translation}
\dim V^{\langle\sigma\rangle}\ge n-1 
\quad\Longleftrightarrow\quad
\dim (\sigma-1)V \le 1.
\end{equation}

\begin{lemma}
\label{lemma:groupactions}
Let $G$ be a group acting linearly on a vector space $V$ of 
dimension $n\in\Zee_{\ge0}$.
Then the following are equivalent:
\begin{enumerate}
\item
$\dim V^{\langle\sigma\rangle}\ge n-1$ for all $\sigma \in G$;
\item
$\dim V^G\ge n-1$ or\/  $\dim V_G\ge n-1$.
\end{enumerate}
\end{lemma}
\begin{proof}
The implication $(2)\Rightarrow(1)$ is immediate, as the inequality
$\dim V_G\ge n-1$
implies that, for all $\sigma\in G$,
we have $\dim (\sigma-1)V \le 1$ and,
by \eqref{invariants-translation}, 
$\dim V^{\langle\sigma\rangle}\ge n-1$.

For $(1)\Rightarrow (2)$, we can assume 
there exists $\sigma\in G$ acting non-trivially on $V$,
and define subgroups $A_\sigma, B_\sigma \subset G$ by
$$
A_\sigma=
  \{\tau\in G\colon V^{\langle\tau\rangle} \supset V^{\langle\sigma\rangle} \}
\qquad\text{and}\qquad
B_\sigma=
  \{\tau\in G\colon (\tau-1)V \subset (\sigma-1)V \}.
$$
The equality $A_\sigma=G$ implies that $V^G=V^{\langle\sigma\rangle}$
has dimension $n-1$, and the equality $B_\sigma=G$ implies that 
$V_G=V/(\sigma-1)V$ has dimension $n-1$.
In order to show that we have one of these equalities, and therefore (2),
we argue by contradiction. 
Assume $A_\sigma$ and $B_\sigma$ are \emph{strict} subgroups of $G$, and
pick $\tau\in G$ outside $A_\sigma\cup B_\sigma$.
Then there exist 
$s\in V^{\langle\sigma\rangle}\setminus V^{\langle\tau\rangle}$
and
$t\in V^{\langle\tau\rangle}\setminus V^{\langle\sigma\rangle}$,
and $(\sigma-1)V$ and $(\tau-1)V$ are \emph{different} 1-dimensional
subspaces of $V$ spanned by $(\sigma-1) t$ and $(\tau-1)s$, respectively.

The subspace $(\tau\sigma-1)V$ is 1-dimensional and spanned by
$(\tau\sigma-1)s=(\tau-1)s$, so it equals $(\tau-1)V$.
It contains $(\tau\sigma-1)t=\tau(\sigma-1)t$, but
since $\tau$ acts on $(\sigma-1)t\notin (\tau-1)V$
by translation along a vector in $(\tau-1)V$,
we have $\tau(\sigma-1)t\notin (\tau-1)V$. Contradiction.
\end{proof}
\noindent
For those who like to think of Lemma \ref{lemma:groupactions}
in terms of matrices, Condition~(1) means that every element of $G$ 
has a matrix representation with respect to a suitable basis that, 
according to \eqref{invariants-translation},
can be given in one of the equivalent forms
\begin{equation}
\label{twoforms}
\left(\hspace{-2mm}
\begin{array}{cc}
\framebox[1.5cm][c]{
\mbox{$\begin{array}{c}
     \\
I_{n-1}\\
\\
\end{array}
$}}
&\hspace{-5mm} 
\begin{array}{c} *\\ \vdots \\ *\end{array}\\
  \begin{array}{ccc}
     0  & \cdots & 0 
  \end{array} &\hspace{-5mm}  *
\end{array}\hspace{-2mm}
\right)
\qquad\text{or}\qquad
\left(\hspace{-2mm}
\begin{array}{cc}
*&\hspace{-5mm}\begin{array}{ccc}
     *  & \cdots & * 
  \end{array}\\
 \begin{array}{c} 0\\ \vdots \\ 0\end{array}
 &\hspace{-5mm}
\framebox[1.5cm][c]{\mbox{$
\begin{array}{c} \\ I_{n-1}\\ \\ \end{array}
$}}
\end{array}\hspace{-2mm}
\right).
\end{equation}
The first form shows $n-1$ linearly independent 
vectors in $V^{\langle\sigma\rangle}$, the second 
starts from a vector spanning $(\sigma-1)V$.
The lemma then states that under this condition, a \emph{single} 
basis for $V$ can be chosen such that either all elements of $G$ 
have a matrix representation of the first form, or they all have 
one of the second form.

\begin{example}
\label{Serre-example}
For an elliptic curve $E/K$, we can 
apply Lemma \ref{lemma:groupactions} to the action 
of the Galois group $G$ of the $\ell$-division field 
of $E$ over $K$ on the 2-dimensional $\FF_\ell$-vector space
$V=E[\ell](\overline K)$ of $\ell$-torsion points of $E$.
In this case the point group $E(k_\gothp)$ at a prime 
$\gothp\nmid \ell$ of good reduction is of order divisible 
by $\ell$ if and only if $\Frob_\gothp\in G$ 
pointwise fixes a 1-dimensional subspace of $V$.
We find that almost all local point groups $E(k_\gothp)$ are 
of order divisible by $\ell$ if and only if the 
Galois representation $\rho_{E/K,\ell}$ of $G_K$ 
on the group of $\ell$-torsion points of $E$ 
can be given in matrix form as 
\[
\rho_{E/K,\ell} \sim 
\begin{pmatrix}
1 & *  \\
0 & *  
\end{pmatrix}
\qquad\text{or}\qquad
\rho_{E/K,\ell} \sim 
\begin{pmatrix}
* & *  \\
0 & 1  
\end{pmatrix}.
\]
In words: either $E(K)$ contains an $\ell$-torsion point,
or it is $\ell$-isogenous over $K$ to an elliptic curve 
with a $K$-rational $\ell$-torsion point. 
Moreover, for $E/K$ of the first kind, with a point $T\in E(K)$ of order $\ell$,
the quotient curve $E'=E/\langle T\rangle$ is of the second kind, with
the dual isogeny $E'\to E$ being the $\ell$-isogeny in question.
This is a well-known fact that occurs as the very first exercise  
in \cite{Serre2}*{p.\ I-2}.
\end{example}

\begin{proof}[Proof of Theorem \ref{thm:lneverprim}]
Let $E/K$ be an elliptic curve, and $P\in E(K)$ 
a non-torsion point that is locally $\ell$-imprimitive.
We define $K_\ell=K(\ell^{-1}P)$ as in Section \ref{S:splittingconditions},
and view $G_\ell=\Gal(K_\ell/K)\subset \GL(V_\ell)$ as 
a group of $\FF_\ell$-linear automorphisms of the 
3-dimensional vector space 
$V_\ell= \langle \ell^{-1}P\rangle/\langle P \rangle$.
As every element of $G_\ell$ occurs as the Frobenius 
of infinitely many primes of good reduction, it follows from
Lemma~\ref{ellprimrootatatell} that all elements of $G_\ell$
leave a 2-dimensional subspace of $V_\ell$ pointwise invariant.
We can now apply our Lemma \ref{lemma:groupactions}
for $n=3$ with $G=G_\ell$ and $V=V_\ell$ to conclude that at least one of the following occurs:
either $G_\ell$ acts trivially on a 2-dimensional subspace of $V_\ell$,
or $G_\ell$ acts trivially on a 2-dimensional quotient space of $V_\ell$.

In the first case, if $U_\ell=E[\ell](\overline K)\subset V_\ell$ is a subspace
with trivial $G_\ell$-action, then $E(K)$ has complete $\ell$-torsion
and Condition B is satisfied.
If $G_\ell$ acts trivially on a different 2-dimensional subspace
$S_\ell\subset V_\ell$, than $S_\ell$ is spanned by a non-zero
vector in $U_\ell\cap S_\ell$ and the non-zero image of a point of infinite order
$Q\in \langle \ell^{-1}P\rangle$ in the  $\FF_\ell$-vector space $V_\ell$.
In other words: $E(K)$ contains a torsion point of order $\ell$ and
a point $Q$ with $\ell Q=mP$ for some $m\in\Zee$ 
that is not divisible by $\ell$.
Writing $am+b\ell=1$ in~$\Zee$, the point $Q'=aQ+bP\in E(K)$ satisfies
$\ell Q'= a\ell Q+b\ell P=amP+b\ell P=P$, so Condition A is satisfied.

In the second case, where $G_\ell$ acts trivially on a 2-dimensional quotient space 
$V_\ell/T_\ell$, it acts on $V_\ell$ by translation along the 1-dimensional 
subspace $T_\ell$. 
We will assume that $G_\ell$ does not act trivially on the subspace $U_\ell$, 
as this would bring us back in the first case, with Condition B holding.
As $U_\ell=E[\ell](\overline K)\subset V_\ell$ is $G_\ell$-stable, we have 
strict inclusions $0\subsetneq T_\ell\subsetneq U_\ell$ of $\FF_\ell[G_\ell]$-modules,
so $T_\ell$ is a $K$-rational subgroup of $E(\overline K)$ of order $\ell$.
The corresponding isogeny $E\to E'=E/T_\ell$ is defined over~$K$,
and identifies the $\FF_\ell[G_\ell]$-module $U_\ell/ T_\ell$, which has trivial $G_\ell$-action,
with the subgroup of $E'(K)$ of order $\ell$ that is the kernel of the isogeny 
$\phi: E'\to E$ dual to $\widehat\phi: E\to E'=E/T_\ell$.
If $Q\in E(\overline K)$ satisfies $\ell Q=P$, then 
$P'=\widehat\phi(Q)$ is in $E'(K)$, as it is the image of any point in the Galois orbit 
$G_\ell\cdot Q\subset Q+T_\ell$.
Moreover, we have $\phi(P')=\phi\widehat\phi(Q)=\ell Q=P$, so Condition C 
of Theorem \ref{thm:lneverprim} is satisfied.

Conversely, each of the Conditions A, B, and C guarantees that $P\in E(K)$ is locally
$\ell$-imprimitive. For $A$ and $B$ this is immediate.
If Condition C holds, we have an $\ell$-isogeny $\phi: E'\to E$ defined over $K$ and
a point $P'\in E'(K)$ with $\phi(P')=P$.
Pick $Q'\in E'(\overline K)$ with $\ell Q'=P'$ and put $Q=\phi(Q')\in E(\overline K)$.
Writing $\widehat\phi: E\to E'$ for the dual isogeny, we are in the situation of Example 
\ref{Serre-example}, and we have 
\[
\ell Q=\phi(\ell Q')=\phi(P')=P\qquad \text{and} \qquad
\widehat\phi Q= \widehat\phi \phi Q'= \ell Q'=P'\in E'(K).
\]
As $Q$ is in the fibre $\widehat\phi^{-1}(P')$, the $G_\ell$-action on $Q \bmod \langle P\rangle\in V_\ell$, 
which is by translation over $\ell$-torsion points, gives rise to a Galois orbit of length dividing $\ell$.
If the length is 1, then Condition A is satisfied, and $P\in E(K)$ is locally $\ell$-imprimitive.
If the length is $\ell$, then $G_\ell$ acts on $V_\ell$ by translation along the $K$-rational subgroup
$T_\ell=\ker\widehat\phi\subset U_\ell=E[\ell](\overline K)$, and the matrix representation of $G_\ell$ on $V_\ell$
with respect to the filtration $T_\ell\subset U_\ell\subset V_\ell$ is
\[
\rho_{E/K,\ell} \sim 
\begin{pmatrix}
* & * & * \\
0 & 1 & 0 \\
0 & 0 & 1
\end{pmatrix}.
\]
By Lemma \ref{ellprimrootatatell}, we find that we have 
$\ell | [E(k_\gothp):\langle {\bar P}\rangle ]$ for every prime 
$\gothp$ of good reduction of~$E$ that is of characteristic different from $\ell$,
so $P\in E(K)$ is locally $\ell$-imprimitive.
\end{proof}
\section{Locally 2-imprimitive points}
\label{S:Locally 2-imprimitive points}
\noindent
A \emph{non-trivial} locally $\ell$-imprimitive point on an elliptic curve $E/K$
is a non-torsion point $P\in E(K)$ for which Condition C of 
Theorem \ref{thm:lneverprim} holds, but not Conditions A or B.
If $P$ is such a point, $E$ admits a $K$-rational $\ell$-isogeny $E\to E'$, 
and the Galois representation $\bar\rho_{E, \ell}$ of $G_K$ on  
$U_\ell=E[\ell](\bar K)$  
is non-trivial, with image contained in a Borel subgroup of $\GL(U_\ell)$.

Let $P\in E(K)$ be a non-trivial locally 2-imprimitive point.
As a Borel subgroup of $\GL(U_2)\cong \GL_2(\FF_2)$ has order 2, 
the representation $\bar\rho_{E, 2}$ is a non-trivial quadratic character,
and as a Weierstrass model for $E$ we can take 
\begin{equation}
\label{00model}
E: y^2=x(x^2+ax+b) \qquad \hbox{with } b, d=a^2-4b\in K^* .
\end{equation}
Here $(0,0)$ is the $K$-rational point of order 2, and $\bar \rho_{E, 2}$ 
corresponds to the $2$-division field of $E$ over $K$, which equals $K(\sqrt d)$.
Addition by $(0,0)$ induces an involution 
\[
(x,y)\mapsto (x_1, y_1)=(b/x, -by/x^2)
\]
on the function field $K(E)=K(x,y)$ of $E$, and the invariant field 
$K(x+x_1, y+y_1)$ is the function field of the
2-isogenous curve $E'=E/\langle(0,0)\rangle$.
Choosing $u=x+x_1+a$ and $v=y+y_1$ as generators for $K(E')$, 
we obtain a Weierstrass model
\begin{equation}
\label{Eprime}
E': v^2=u(u^2-2au+d)  \qquad \hbox{with } d, d'=(-2a)^2-4d=16b\in K^*
\end{equation}
for $E'$ that is of the same form \eqref{00model}, 
and an explicit 2-isogeny $\varphi: E\to E'$ given by
\begin{equation}
\label{2-isogeny}
\varphi: (x,y)\longmapsto (u, v) = (x+x_1+a, y+y_1)=
\left( \frac {y^2}{x^2}, (1-\frac{b}{x^2})y \right) .
%= \left( \frac {x^2+ax+b}x, (1-\frac{b}{x^2})y \right).
\end{equation}
An affine point $(u,v)\in E'(K)$ different from $(0,0)$ is in the image of $E(K)$ under this isogeny 
if and only if $u\in K^*$ is a square.
The point $(0,0)$ is in the image if and only if $d$ is a square in $K^*$, 
which amounts to saying that $E(K)$ has full 2-torsion.
This not the case for our $E$. 

As $E'$ is again of the form \eqref{00model}, 
one sees that the isogeny $\widehat\varphi: E'\to E$ dual to $\varphi$ is given by
\begin{equation}
\label{dual-2-isogeny}
\widehat\varphi: (u,v)\longmapsto
\left( \frac {v^2}{4u^2}, (1-\frac{b}{u^2})\frac{v}{8} \right) .
\end{equation}

\begin{proof}[Proof of Theorem \ref{thm:2neverprim}]
Let $E/K$ be an elliptic curve with $\#E(\overline K)=2$.
We take $(0,0)$ as the point of order 2 in a Weierstrass model of $E$
in order to obtain an equation $E: y^2=x(x^2+ax+b)$ as in \eqref{00model}, 
with $d=a^2-4b\in K^*\setminus {K^*}^2$.     
Any quadratic twist of $E$ over $K$ is of the form 
\[
E_D: y^2=x(x^2+aDx+bD^2)
\]
for some $D\in K^*$ that we may still rescale by squares in $K^*$, 
and we can define the 2-isogeny
$\varphi: E_D \to E_D'=E/\langle (0,0)\rangle$ as above 
by replacing $(a,b)$ in \eqref{00model}, \eqref{Eprime}, 
and~\eqref{2-isogeny} by %$(\tilde a, \tilde b)=
$(aD, bD^2)$.
Any point $P\in E_D(K)$ satisfying Condition C from 
Theorem \ref{thm:lneverprim} is in the image $\widehat\varphi(E_D'(K))$
of the isogeny $\widehat\varphi: E_D'\to E_D$ dual to $\varphi$.
This means that its $x$-coordinate is a square in $K^*$, which we can take to be 1
after rescaling the model of $E$ over $K$.
Thus, the twists that are relevant for us are those for which the point
$P=(1,\pm Y)\in E_D(\bar K)$ is $K$-rational.

We want $(D,Y)$ to be a $K$-rational point on
the conic $Y^2=1+aD+bD^2$ different from $(0,\pm 1)$.
Such points are obtained as the second point of intersection 
of this conic with the line $Y-1=\lambda D$ through $(0,1)$ with slope $\lambda \in K$.
We find that the twists $E_\lambda=E_{D_\lambda}$ of $E$ by 
\begin{equation}
\label{Dlambda}
D_\lambda=\frac{a-2\lambda}{\lambda^2-b} \qquad     
\text{with }\lambda\in K\setminus
\{a/2,\pm\sqrt b\}
\end{equation}
come by construction with a $K$-rational point
\begin{equation}
\label{Plambda}
P_\lambda=(1,Y_\lambda)=
\left(1,\frac{\lambda^2-a\lambda+b}{\lambda^2-b}\right)
\in E_\lambda(K).
\end{equation}
We can find $K_2=K(\frac{1}{2} P_\lambda)$
by solving $2Q=\widehat\varphi(\varphi Q) = P_\lambda$ in $E_\lambda(\overline K)$.
The equation $\widehat\varphi(u,v)=P_\lambda$
has 2 solutions in $E'_\lambda(K)$, since we chose for 
the first coordinate of $P_\lambda$ the square value 1. 
By \eqref{dual-2-isogeny}, these are the $K$-rational points $(u,v)$ different from $(0,0)$ 
that have $v^2=4u^2$ and satisfy the equation 
\[
v^2= u(u^2-2aD_\lambda u+ dD_\lambda^2)
\]
defining $E'_\lambda$. 
With $d=a^2-4b$, the resulting equation
\[
u^2-2aD_\lambda u -4u + dD_\lambda^2 = (u-aD_\lambda-2)^2 - 4 (1+aD_\lambda+bD_\lambda^2)=0
\]
for $u$ yields two solutions
$
u_1, u_2 = aD_\lambda+2 \pm Y_\lambda  \in K
$
with product $u_1u_2=dD_\lambda^2$.
Writing $D_\lambda$ and $P_\lambda$ as in 
\eqref{Dlambda} and \eqref{Plambda},
we find $aD_\lambda+2 - Y_\lambda=d/(\lambda^2-b)$, so
the minimal extension $K\subset K_2$ for which the 
corresponding points are in $\varphi[E(K_2)]$ is 
\[
K_2 = K(\textstyle\frac{1}{2} P_\lambda)= K(\sqrt {u_1}, \sqrt{u_2})
     = K(\sqrt d, \sqrt{\lambda^2 - b}).
\]
If we avoid the values $\lambda\in K$
for which $\lambda^2 - b$ is either a square or $d$ times a 
square in $K$--this includes the 3 values excluded in \eqref{Dlambda}--then 
$K_2$ is a biquadratic extension of $K$ which, unsurprisingly, 
has the 2-division field $K(\sqrt d)$ of $E/K$
as one of its quadratic subextensions.
For these~$\lambda$, our matrix representation from
\eqref{mod-ell-rep} and~\eqref{matrixrep} 
of $G_2=\Gal(K_2/K)$ becomes
\[
\Gal(K_2/K)= 
%\Gal(K(\textstyle\frac{1}{2}P_\lambda)/K)=
\left\{ \begin{bmatrix}
1 & x & y \\
0 & 1 & 0 \\
0 & 0 & 1 
\end{bmatrix}: x,y\in\FF_2  \right\},
\]
which implies that $P_\lambda\in E_\lambda(K)$ is globally 2-primitive, but locally $2$-imprimitive. 

There is still the possibility that $P_\lambda$, though globally 2-primitive, is a torsion point of even order $m>2$.
Examples: 
the point $(1,-3)$ on $y^2=x(x^2-7x+3)$ has order 4, and 
the point $(1,1)$ on $y^2=x(x^2+3x-3)$ has order 6.
However, for fixed~$K$ there are only finitely many possibilities
for $m$ by Merel's theorem. 
For every given even $m$, the point $P_\lambda$ is of order $m$ 
if and only if the $m$-th division polynomial
$\psi_m(x)=y^{-1}f(x, E_\lambda)$ vanishes at $x=1$.
This happens for only finitely many values of~$\lambda$,
as $f(1,E_\lambda)$ is a non-constant rational function of $\lambda$ if we fix $a,b \in K$.
\end{proof}

\noindent 
Our understanding of local 2-imprimitivity is more or less complete, as
every non-trivial 2-locally imprimitive point on an elliptic curve arises
as in the construction in the proof of Theorem \ref{thm:2neverprim}.
Indeed, the hypothesis $\#E[2](\overline K)=2$ implies that $E$ has a model
as in \eqref{00model}, and as the $x$-coordinate of a point $P$ satisfying Condition C 
of Theorem \ref{thm:lneverprim} is a square, the model can be scaled over $K$ to have
$P=(1,y)\in E(K)$.

\section{Locally 3-imprimitive points}
\label{S:Locally 3-imprimitive points}

\noindent
By Theorem \ref{thm:lneverprim}, every pair $(E,P)$ of an
elliptic curve $E/K$ with a non-trivial locally $\ell$-imprimitive point $P\in E(K)$ arises as the $\ell$-isogenous image of a $K$-rational 
curve-point-pair $(E',P')$ for which the kernel $E'\to E$ 
is generated by an $\ell$-torsion point $T\in E'(K)$.
In this situation,  
the Galois representations of $G_K$ on the $\ell$-torsion subgroups  
of $E'$ and $E=E'/\langle T\rangle$ are, 
with respect to a suitable basis, of the form 
\begin{equation}
\label{mod-ellreps}
\rho_{E'/K,\ell} \sim 
\begin{pmatrix}
1 & *  \\
0 & \omega_\ell  
\end{pmatrix}
\qquad\text{and}\qquad
\rho_{E/K,\ell} \sim 
\begin{pmatrix}
\omega_\ell & *  \\
0 & 1  
\end{pmatrix}.
\end{equation}
Here $\omega_\ell$ is the cyclotomic character corresponding to 
the extension $K\subset K(\zeta_\ell)$.
For $\ell > 2$, the $K$-rationality of $\ell$-torsion points of $E$
is not preserved under twisting of $E$, so there is no direct analogue of
Theorem \ref{thm:2neverprim} for $\ell\ne2$.
In this section we focus on the case $\ell=3$.
\begin{lemma}
\label{K3-cuberootdisc}
Let $E/K$ be an elliptic curve of discriminant $\Delta_{E}$
for which the Galois representation
$\rho_{E/K, 3}$ on $U_3=E[3](\overline K)$ is of one of the two forms in 
\eqref{mod-ellreps}.
Then the $3$-division field of $E$ over $K$ equals the splitting field 
of the polynomial $X^3-\Delta_{E}$.
\end{lemma}
\begin{proof}
Let $H_3=\Gal(K(E[3])/K)\subset \GL(U_3)\iso \GL_2(\FF_3)$ be the image of 
$\rho_{E/K, 3}$, and denote by $\psi_3=\prod_{i=1}^4 (X- x_i)\in K[X]$ 
the 3-division polynomial of $E$ over~$K$.
The quartic polynomial $\psi_3$ comes with a Galois resolvent
$\delta_3\in K[X]$, a cubic having 
\[
\alpha_1=x_1x_2+x_3x_4, \qquad
\alpha_2=x_1x_3+x_2x_4, \qquad
\alpha_3=x_1x_4+x_2x_3
\]
as its roots. 
Under the permutation action of $\GL(U_3)/\langle -1\rangle=S_4$ on the roots of~$\psi_3$,
the normal subgroup $V_4\triangleleft S_4$ of order 4 fixes 
each of these 3 roots $\alpha_i$.
The two natural surjections of Galois groups
\begin{equation}
\label{Galoissurj}
H_3 \to \Gal(\psi_3)\to \Gal(\delta_3)
\end{equation}
are isomorphisms as they
arise as a restriction to suitable subgroups of the generic
group theoretical maps 
\[
\GL(U_3)\to \GL(U_3)/\langle-1\rangle=S_4 \to S_4/V_4=S_3.
\]
More precisely, the first surjection in \eqref{Galoissurj}
is injective as we have $-1\notin H_3$
for $H_3$ as in \eqref{mod-ellreps},
and the second is because we have $\Gal(\psi_3)\cap V_4= 1 $ in $S_4$:
the $x$-coordinate of the 3-torsion point spanning the Galois invariant subspace corresponding to the first column of the matrices 
is fixed by $\Gal(\psi_3)$.
Viewing $H_3$ as $\Gal(\delta_3)$, we may finish the proof by quoting
a classical formula \cite{Serre}*{p.~305} that expresses the three 
cube roots of $\Delta_E$ as $b_4-3\alpha_i$ ($i=1,2,3$), with
$b_i\in K$ a coefficient from the Weierstrass model of $E$.
This yields $K(E[3])=\Split_K(g)=\Split_K(X^3-\Delta_{E})$.
\end{proof}
\noindent
From Lemma \ref{K3-cuberootdisc}, we see that
for the representations in \eqref{mod-ellreps}, the subgroup
%$\begin{pmatrix} 1&*\\0&1\end{pmatrix}$ 
%${1\ *\atopwithdelims() 0\ 1}$
$\binom{1\ *}{0\ 1}$
corresponds to the extension 
$K(\zeta_3)\subset K(\zeta_3,\root 3 \of \Delta)$ 
for the discriminant values $\Delta=\Delta_{E'}$ and $\Delta_E$.

We can write the curve $E'$ in Deuring normal form \cite{Husemoller}*{p.\ 89} as
\begin{equation}
\label{Deuringnf-ab}
E': y^2+axy+by=x^3 \qquad \hbox{with } 
    a,b \in K \hbox{ and } \Delta_{E'}=b^3(a^3-27b)\in K^*.
\end{equation}
Here $T=(0,0)\in E'(K)$ is the point of order 3, and
the quotient curve $E=E'/\langle T\rangle$ has 
Weierstrass equation 
\begin{equation}
\label{3-isogcurve-ab}
E: y^2+axy+by=x^3-5abx-(a^3+7b)b,
\end{equation}
with the explicit formula for the 3-isogeny 
$\varphi_3: E'\to E=E'/\langle T\rangle$  being given by 
\begin{equation}
\label{explicit3isog}
\varphi_3(x,y)= \left(
\frac{x^3+abx+b^2}{x^2},
    \frac{y(x^3-abx-2b^2)-b(ax+b)^2}{x^3}\right).
\end{equation}
As $E$ has discriminant $\Delta_E=b(a^3-27b)^3$, the 
3-division field $K(E[3])$ over $K$ equals $K(\zeta_3)$  
if and only if $b\in K^*$ is a cube and different from $(a/3)^3$.

If $a$ is non-zero, we can rescale 
$(x, y)\mapsto (a^2x, a^3y)$ and simplify \eqref{Deuringnf-ab} to 
\begin{equation}
\label{Deuringnf}
E_b': y^2+xy+by=x^3.
\end{equation}
For $K$ a number field, we have infinitely many pairwise different
3-isogenous images $(E_b, P_b)=\varphi_3[(E_b', P_b')]$ for which
$P_b\in E_b(K)$ is non-trivially locally 3-imprimitive.
\begin{theorem}
\label{ell=3}
For $K$ a number field, take $b=b(X)=(1-X-X^2)/X\in K(X)$ and 
define the associated elliptic curve over $K(X)$ as
$$
E_{b}: y^2+xy+by=x^3-5bx-(1+7b)b.
$$  
Then for infinitely many $t\in K^*$, the specialization $E_{b(t)}$
of $E_b$ is an elliptic curve over $K$ for which
\begin{equation}
\label{pointPt}
    \left(\frac{(t^2+t)(t^2-1)+1}{t^2},
    \frac{(1-t^2)\left(t^4+2 t^3+t-1\right)}{t^3}\right)
\in E_{b(t)}(K)
\end{equation} 
is a non-trivial locally 3-imprimitive point.
\end{theorem}
\begin{proof}
For the curve $E'_b$ in \eqref{Deuringnf} with $b\in K(X)$ as defined, 
the point $P'_b=(1, X)$ lies on $E'_b$ by the very choice of $b$:
it satisfies $X^2+X+bX=1$.
Under the 3-isogeny $\varphi_3: E'_b\to E_b$ from 
\eqref{explicit3isog} to the curve
$E_b$ obtained by putting $a=1$ in \eqref{3-isogcurve-ab},
it is mapped to 
$$
P_b=(1+b+b^2, (1-b-2b^2)X-b(1+b)^2)\in E_b(K(X)).
$$
Under the specialization $X=t\in K^*$, we obtain a point $P_{b(t)}$,
on the curve $E_{b(t)}$ defined over $K$ 
that is given by \eqref{pointPt}.
We are only interested in those 
specializations for which $E_{b(t)}$ is an elliptic curve.
As these are all $t\in K^*$ for which $b(t)\notin \{0, 1/27\}$, at most
4 `bad' values of $t$ are excluded.
Also, by the same argument as we gave for $P_\lambda$ in the case $\ell=2$,
there are only finitely many $t\in K^*$ for which $P_t$ is torsion point. 
These finitely many $t$ we also exclude as `bad' values.

We saw already that $E_b$ has 3-division field 
$K(\zeta_3, \root 3 \of b)$, and an explicit computation 
shows that the 3-division field of the point $P_b\in E_b(K(X))$ equals 
\[
K(X)(\textstyle\frac13 P_b)=K(\zeta_3, \root 3 \of b, \root 3 \of X)=
K(\zeta_3, \root 3 \of X, \root 3 \of {X^2+X-1}).
\]
Over $K(\zeta_3, X)$, the elements $X$ and $X^2+X-1$ have 
`independent' cube roots -- it suffices look at their ramification locus.
It follows that the Galois group of 
the 3-division field of $P_b$ over $K(X)$
may be described as
\[
\Gal(K(X)(\textstyle\frac13 P_b)/K(X))= 
\left\{ \begin{bmatrix}
\omega_3 & x & y \\
0 & 1 & 0 \\
0 & 0 & 1 
\end{bmatrix}: x,y\in\FF_3  \right\},
\]
with $\omega_3$ denoting the $\FF_3^*$-valued character corresponding to
the (possibly trivial) extension $K(X)\subset K(X, \zeta_3)$.
By Hilbert irreducibility, it follows that for infinitely many
$t\in K^*$ outside the finite set of `bad' values, 
the 3-division field of the point $P_{b(t)}\in E_{b(t)}(K)$ 
has the `same' Galois group over $K$,
making it into a point that is globally 3-primitive, but locally $3$-imprimitive.
As $E_{b(t)}$ does not have complete 3-torsion over $K$, we conclude that the point $P_{b(t)}$ given in \eqref{pointPt}
is a non-trivial locally 3-imprimitive point.
\end{proof}
\noindent
\begin{remark}
The construction in the proof of Theorem \ref{ell=3} excludes all specializations for which $b=b(t)\in K^*$ is a cube and 
the elliptic curve in \eqref{Deuringnf} has 3-division field $K(\zeta_3)$.
In this special case, we can also equip $E=E_b$  
with a non-trivial locally 3-imprimitive point
for infinitely many $b\in {K^*}^3$.
We first write $b=c^{-3}$ and transform the curve under
$(x,y)\mapsto (c^{-2}x, c^{-3}y)$ into $E':y^2+cxy+y=x^3$.
As in the previous case, we have $P'=(1,t)\in E'(K)$ 
for $c=(-t^2-t+1)/t$, and the image of $P'$ under the map 
\[
\varphi: E'\to E= E'/\langle(0,0)\rangle: y^2+cxy+y=x^3-5cx-(c^3+7)
\]
is the point $P=\varphi_3(P')=((-t^2 + t + 1)/t, (t^2 - 1)/t^2)\in E(K)$,
for which the 3-division field $K(\frac13P)$ is equal to  
$K(\zeta_3,\sqrt[3]t)$.
For almost all $t\notin {K^*}^3$, this makes $P$ into a 
globally 3-primitive point that is locally $3$-imprimitive. 
\end{remark}
\noindent
We conclude our discussion for $\ell=3$ with the remaining case in which
the curve $E'$ in \eqref{Deuringnf-ab} has $a=0$.
In this case $E'$ has $j$-invariant 0, and writing $c=b/2$
we may rescale the equation by $y\mapsto y-c$ to the more 
familiar shape $E': y^2=x^3+c^2$, 
with 3-torsion point $T=(0,c)$ and CM by $\Zee[\zeta_3]$.
We equip $E'$ with a $K$-rational point $(c, c\sqrt{c+1})$
by putting $c=s^2-1$ with $s\in K$.
This leads to a 1-parameter family of 3-isogenies 
\begin{align*}
    \varphi_3: E'_s: y^2= x^3 + (s^2-1)^2 &\tto  E_s: y^2= x^3 - 27 (s^2-1)^2 \\
    P'_s=(s^2-1, s(s^2-1))&\longmapsto P_s=(s^2+3, s(s^2-9))
\end{align*}
between CM-curves with $j$-invariant 0.
In this case the 3-division field of $E_s$ over~$K$ is 
$K(\zeta_3, \root 3 \of {2(s^2-1)})$, and the 3-division field of $P_s$ over $K$ equals
\[
K(\zeta_3, \root 3 \of {2(s^2-1)}, \root 3 \of {4(s+1)}).
\]
Again, for $s\in K$ outside a thin set, the point $P_s\in E_s(K)$
is globally 3-primitive but locally 3-imprimitive.
\section{Further examples}
\label{S:Further examples}

\noindent
Over $K=\Que$, non-trivial locally $\ell$-imprimitive points 
can only occur for primes $\ell\le 7$.
Examples for $\ell=5$ and also $\ell=7$
can be found by the techniques that we employed for $\ell=3$, but the 
formulas and resulting curves rapidly become less suitable for presentation on paper.

\subsection{Curves with a locally 5-imprimitive point}
In this case, we start from Tate's normal form
\begin{equation}
\label{Tate-5}
E': y^2+(1-c)xy-cy=x^3-cx^2
\end{equation}
that parametrises elliptic curves with $(0,0)$ as a $5$-rational point
(see Kulesz~\cite{Kulesz}). 
It has further points $(0,c)$, $(c,0)$ and $(c, c^2)$ of order 5,
and its discriminant equals
$\Delta_{E'}=c^5 \left(c^2-11 c-1\right)$.
%
%$$j_E=\frac{\left(c^4-12 c^3+14 c^2+12 c+1\right)^3}{c^5 \left(c^2-11 c-1\right)}$$
Using V\'elu's formula \cite{Velu} or invoking Pari-GP,
we compute the Weierstrass equation for the 5-isogenous curve 
$E=E'/\langle(0,0)\rangle$
\begin{align*}
E=E_c: \quad y^2+&(1-c)xy-cy= \\
&x^3-cx^2 -5 c (c^2+2 c-1)x-
c (c^4+ 10 c^3 - 5 c^2+ 15 c -1),
\end{align*}
and also the explicit 5-isogeny $\varphi_5: E'\rightarrow E$. 
The discriminants involved are
$\Delta_{E'}=c^5(c^2-11c-1)$ and $\Delta_{E}=c(c^2-11c-1)^5$, 
much like we saw for $\ell=3$.
The 5-torsion representations
of $E'$ and $E$ are as in \eqref{mod-ellreps}, 
and even though the proof of Lemma \ref{K3-cuberootdisc}
for $\ell=3$  does not generalize to $\ell=5$, we found by 
a direct calculation that the 5-division fields are 
$K(E'[5])=K(\zeta_5,\sqrt[5]{c^2-11c-1)})$ and
$K(E[5])=K(\zeta_5,\sqrt[5]{c})$: 
generated over $K(\zeta_5)$ by the 5-th root of the discriminant.

We can equip $E'$ with a $K$-rational point $P_t'=(t,t)$ by 
putting $c=t(2-t)$, and compute its image 
$P_t=\varphi_5(P')\in E_{t(2-t)}(K)$ as

\medskip
\centerline{$
P_t= \big(\frac{2 t^4-8 t^3+11 t^2-6 t+2}{(t-1)^2} ,
-\frac{t^8-7 t^7+19 t^6-23 t^5+4 t^4+23 t^3-31 t^2+19 t-4}{(t-1)^3}\big).
$}
\medskip
\noindent
The corresponding 5-division field of $P_t$ is
$$
\textstyle
K(\frac15 P_t)=K(\zeta_5, \root 5 \of t, \root 5 \of {t-2}).
$$
If this is an extension of degree 5 of 
$K(\zeta_5, \root 5 \of c)=K(\zeta_5, \root 5 \of {t(t-2)})$, then 
$P_t$ is a globally 5-primitive but locally 5-imprimitive point 
in $E_{t(2-t)}(K)$.

\begin{example}
Take $K=\Que$.
For $t=1$ the point $P_t$ above is the zero point as 
$P'_t$ is 5-torsion,
for $t=2$ we have $c=0$ and $E$ is singular, while for
$t=3$ and 4 we encounter the `accidents' $t=-c$ and $tc=2^5$
leading to points $P_t\in 5E_{t(2-t)}(\Que)$.
For $t=5$ we obtain the point $P_5=(497/16, -73441/64)$ on 
\[
E_{-15}: y^2+16xy+15y=x^3+15x^2+14550x+232860,
\]
which is the curve 5835.c2 in the LMFDB-database.
Note that for $c=-15$ we have 
\[
c(c^2-11c-1)=-5835=-3\cdot5\cdot389.
\]
The locally 5-imprimitive point $P_5$ is a generator of 
$E_{-15}(\Que)\iso\Zee$.
In fact, $P_t$ will be globally 5-primitive but locally 5-imprimitive
in $E_{t(2-t)}(\Que)$ for all $t\in\Zee_{\ge5}$ that are not a fifth power
or a fifth power plus 2, as for these $t$ the subgroup of $\Que^*/{\Que^*}^5$
generated by $t$ and $t-2$ has order 25.
\end{example}

\subsection{Curves with a locally 7-imprimitive point}
Again we start from the Tate's normal equation
$$E': y^2+(1-c)xy-by=x^3-bx^2$$
but now we do not impose $b=c$ as for $\ell=5$, but instead
$$c=d^2-d\quad\text{ and } b=d^3-d^2.$$
The curve $E'=E_d'$ parametrizes \cite{Kulesz} elliptic curves 
with $(0,0)$ as point of order 7.
%with $$\Delta_E=(d-1)^7 d^7 \left(d^3-8 d^2+5 d+1\right),$$ $$j_E=\frac{((1 - d + d^2)^3 (1 + 5 d - 10 d^2 - 15 d^3 + 30 d^4 - 11 d^5 +   d^6)^3)}{((-1 + d)^7 d^7 (1 + 5 d - 8 d^2 + d^3))}$$
A Weierstrass equation for the 7-isogenous curve $E=E'/\langle(0,0)\rangle$ is
$$E: y^2+(1-c)xy-by=x^3-bx^2-5 \left(2 b^2+b \left(c^2-3 c-2\right)+c \left(c^2+4 c+1\right)\right)x$$
$$-b^2 \left(12 c^2+c+24\right)-6 b^3+b \left(-c^4+9 c^3+46 c^2+24 c+2\right)-$$
$$c \left(c^4+16 c^3+36 c^2+16 c+1\right).$$  
It has discriminant $\Delta_E=d(d-1)(d^3-8d^2+5d+1)^7$, and  
this time we find its 7-division field to be 
$
K(E[7])=K(\zeta_7,\sqrt[7]{d(d-1)^2}).
$ 

We equip $E'$ with a $K$-rational point $P_t'=(d^2t,d^3t)$
by putting $d=d(t)=({t+1)/(t^2-t+1})$.
The image of $P_t'$ under the 7-isogeny $E'\to E$ is
$$
P_t=\left(-\frac{C(t)}{(2 t-1)^2 (t-1)^2 \left(t^2-t+1\right)^4},\frac{D(t)}{(t-1)^3 (2 t-1)^3 \left(t^2-t+1\right)^6}\right)\in E(K)
$$ 
for certain polynomials $C(t)$ and $D(t)$ in $\Zee[t]$ 
of degree 12 and 18.
%$C(t)=t^{12}-21 t^{11}+106 t^{10}-273 t^9+405 t^8-327 t^7+79 t^6+114 t^5-165 t^4+126 t^3-60 t^2+16 t-2$
%and
%$D(t)=t^{18}-23 t^{17}+162 t^{16}-599 t^{15}+1544 t^{14}-3222 t^{13}+5476 t^{12}-7039 t^{11}+6330 t^{10}-3630 t^9+814 t^8+1128 t^7-2306 t^6+2576 t^5-1893 t^4+916 t^3-283 t^2+51 t-4.$
In terms of~$t$, the 7-division field is
$K(E[7])=K(\zeta_7, \root 7 \of {t^2(t+1)(t-2)^2(t^2-t+1)^4})$, 
and the 7-division field of $P_t$ is
\[
K(\textstyle \frac17 P_t)=
K(E[7])\left(\root 7\of {\frac{t(t+1)}{t-2}}\right)=
K\left(\zeta_7, \root 7 \of {\frac{t(t^2-t+1)}{(t+1)}},
           \root 7 \of {\frac{t(t+1)}{(t-2)}}\right).
\]
The point $P_t$ is a globally 7-primitive but 
locally 7-imprimitive point when the extension
$K(\zeta_7)\subset K(\frac17 P_t)$ has its generic degree $7^2$.

\begin{example}
Take $K=\Que$.
For $t=1$ the point $P_t$ above is the zero point as $P'_t$ is 
7-torsion, and for $t=2$ the curve $E'$ is singular.
For $t=3$ and $d=\frac47$ however we obtain the point 
$P_3=(286019/490^2, 15951227/490^3)$ on 
\[
\textstyle
E: y^2+\frac{61}{7^2}xy+\frac{48}{7^3}y=
x^3+\frac{48}{7^3}x^2-\frac{774780}{7^7}x-\frac{1047829260}{7^{11}},
\]
which is the curve 20622.j1 with minimal model 
\[
E_0: y^2+xy=x^3-5455771x-5039899603, 
\]
in the LMFDB-database.
Our locally 7-imprimitive point $P_3$ is a generator of $E(\Que)\iso\Zee$.
On $E_0$ the corresponding generator is $(328219/10^2, 109777927/10^3)$. 
\end{example}
\section{A composite level obstruction}
\label{S:A level 6 obstruction}

So far we have focused on non-trivial obstructions to
local primitivity at prime level $\ell$, as this is a new
phenomenon in the elliptic primitive root case III that 
does not arise in the multiplicative primitive root case I 
and the cyclic reduction case II.

In all three cases, there exist obstructions of different nature 
at composite levels that arise from the \emph{entanglement} between
finitely many of the corresponding division fields $K_\ell$.
These obstructions do not arise over $K=\Que$, and most examples 
in the cases I and II are created by base changing to a well-chosen finite extension of 
the fields of definition $\Que(x)$ and $\Que(j_E)$.
Again, case III is different here, as entanglement obstructions
already occur over $\Que$.
In this Section we construct a level 6 obstruction.

Let $E/K$ be an elliptic curve with $\#E[2](K)=2$, and 
$P\in E(K)$ a point of infinite order.
Then the 2 division field $K(E[2](\overline K))$ is a quadratic extension of $K$.
Assume that the 2-division field $K(\frac12 P)$ of $P$ is of maximal degree 4 over it.
Then $G_2=\Gal(K(\frac12 P)/K)$ is a dihedral group of order 8 for which 
the matrix representation \eqref{matrixrep} 
on $V_2=\langle\frac12 P\rangle/\langle P\rangle$ has the form
\begin{equation}
\label{G2}
G_2=
\left\{ \begin{bmatrix}
1 & a & c \\
0 & 1 & b \\
0 & 0 & 1 
\end{bmatrix}: a,b,c \in\FF_2  \right\}\subset \GL_3(\FF_2).
\end{equation}
There is a unique subfield of $L\subset K(\frac12 P)$ with Galois group 
over $K$ isomorphic to the Klein 4-group $V_4=C_2\times C_2$, 
and we can view $a$ and $b$ in the
matrix representation \eqref{G2} of $G_2$ as $\FF_2$-valued quadratic characters 
on $G_2$ that generate the character group of 
the quotient $\Gal(L/K)\iso V_4$ of $G_2$.

For a prime $\gothp\nmid 2\Delta_E$, the point $P$ generates 
a subgroup of odd index in $E(k_\gothp)$ if and only if for its
Frobenius $\Frob_{\gothp,2}\in G_2$, viewed as a matrix as in 
\eqref{G2}, the endomorphism $(\Frob_{\gothp,2}-\id_2): V_2\to V_2$ 
has $\FF_2$-rank at least 2  (Lemma \ref{ellprimrootatatell}).
We obtain the criterion
\begin{equation}
\label{ab-cond}
2\nmid [E(k_\gothp):\langle \bar P\rangle]
\quad \Longleftrightarrow \quad 
a(\Frob_{\gothp,2})=b(\Frob_{\gothp,2})=1\in \FF_2.
\end{equation}
More precisely, $a(\Frob_{\gothp,2})=1$ implies that $E(k_\gothp)$ does not have full 2-torsion, 
and $b(\Frob_{\gothp,2})=1$ implies that $P$ is not only not in $2E(k_\gothp)$,
but also not a 2-isogenous image as in Condition C of Theorem \ref{thm:lneverprim}.

Suppose further that $E$ has a $K$-rational 3-torsion subgroup $T$, and let
$\varphi_3:E'\to E$ be the isogeny dual to the quotient map $\phi:E\to E'=E/T$.
Assume that the point $P$ is in $\varphi_3[E'(K)]$ but not in $3E(K)$.
Then the 3-division field $K(\frac13 P)$ of~$P$ has Galois group $G_3=\Gal(K(\frac13 P)/K)$
for which the matrix representation 
on $V_3=\langle\frac13 P\rangle/\langle P\rangle$ will `generically' be the group
\begin{equation}
\label{G3}
G_3=
\left\{ \begin{bmatrix}
d & e & f \\
0 & g & 0 \\
0 & 0 & 1 
\end{bmatrix}: \quad d,g \in \FF_3^*,  e,f \in\FF_3 \right\}\subset \GL_3(\FF_3)
\end{equation}
of order 36.
In this case $d$ and $g$ can be viewed as a quadratic characters $G_3\to \FF_3^*$, 
and another application of Lemma \ref{ellprimrootatatell} shows that for primes 
$\gothp\nmid 3\Delta_E$, we have 
\begin{equation}
\label{g-cond}
g(\Frob_{\gothp,3})=1\in \FF_3^*
\quad \Longrightarrow \quad 
3 | [E(k_\gothp):\langle \bar P\rangle].
\end{equation}
Thus, for primes $\gothp\nmid 6\Delta_E$, a necessary condition for $\bar P \in E(k_\gothp)$ to 
be an elliptic primitive root is that the three quadratic characters $a$, $b$ and $g$
occurring in \eqref{ab-cond} and \eqref{g-cond} do not vanish on the Frobenius automorphism
of $\gothp$ in $K_6=K(\frac16P)$.
In other words: the prime $\gothp$ has to be inert in the quadratic extensions 
$K_a$, $K_b$ and $K_g$ of~$K$ corresponding to these 3 characters.

Primes $\gothp$ satisfying the condition above exist if the quadratic extensions
$K_a$, $K_b$ and $K_g$ are linearly disjoint over $K$, but \emph{not} if they are the three
quadratic subfields of a $V_4$-extension $K\subset K_a K_b K_g$.
In the latter case, we have a splitting obstruction to local primitivity of $P$ in $K_6$ that does not 
exist in one of the smaller fields $K(\frac12 P)$ or $K(\frac13 P)$: it has level 6, but not 2 or 3,
making it an obstruction caused by \emph{entanglement} of division fields.
\begin{example}
An example is provided by the elliptic curve $E/\Que$ with label
12100.j1 in the LMFDB data base.
The curve $E$ has discriminant
$$\Delta_E=2^4\cdot5^9\cdot11^6,
$$
and if we take $(0,0)$ to be its unique $\Que$-rational 2-torsion point
it has Weierstrass model
$$
E: y^2=
x^3 + 605x^2 -3025x.$$
For this curve we have 
$E(\Que)=\langle T_2\rangle\times\langle P\rangle \iso \Zee/2\Zee\times \Zee$
with $T_2=(0,0)$ of order~$2$ and 
$P=(\frac{-13475}{36}, \frac{1249325}{216})$ 
a generator of infinite order.
We have $K_a=\Que(E[2])=\Que(\sqrt{\Delta_E})=\Que(\sqrt 5)$, and over this field
the 2-division field $\Que(\frac12P)$ of $P$ is the $V_4$ extension
$$\Que(E[2])=\Que(\sqrt{5})\subset\Que(\textstyle\frac12P)= 
\Que(\sqrt{5},\sqrt\pi, \sqrt{\bar\pi})
$$
generated by the square roots of $\pi=3+2\sqrt5$ and its conjugate.
From $\pi\bar\pi=-11$ we see that $\Que(\frac12P)$ is cyclic of degree 4 over $\Que(\sqrt{-55})$, and that
we have $K_b=\Que(\sqrt{-11})$.

As $E$ acquires a 3-torsion point $T_3=(\frac{55}{3}, \frac{275}{9}\sqrt{165})$ 
over the quadratic field $\Que(\sqrt{165})=\Que(\sqrt{-3\cdot-55})$ 
that generates a $\Que$-rational torsion subgroup of order~3,
the 3-division field of $E$ has quadratic subfields $K_d=\Que(\sqrt{165})$ and $K_g=\Que(\sqrt{-55})$, making
$K_g$ the third quadratic subfield in the $V_4$-extension $\Que\subset K_aK_b$.
Over the full 3-division field of $E$, the 3-division field of $P$ is the cubic extension
$$
\Que(E[3])=\Que(\sqrt{-3}, \sqrt{-55},\sqrt[3]{2})
\subset\Que(\textstyle\frac13P)=
\Que(E[3], \sqrt[3]{\alpha})
$$
generated by a cube root of an element $\alpha=(3+\sqrt{-55})/2\in K_g$ of norm 16, which shows that
its Galois group over $\Que$ is the group $G_3$ in \eqref{G3}.
We conclude that $P$ is a locally never-primitive point of $E(\Que)$ as the index of $\langle\bar P\rangle$
in $E(\FF_p)$ is always divisible by 2 or 3.
\end{example}
\noindent
An upcoming paper will have further details on obstructions 
to primitivity of composite level, and on how to find explicit examples.

\end{document}